\documentclass[12pt]{article}

\usepackage{graphicx}
\usepackage{subfigure}
\usepackage{algorithmic}
\usepackage{algorithm}

\usepackage{amsthm,amsmath}
\newtheorem{thm}{Theorem}
\numberwithin{thm}{section}

\newtheorem{lem}[thm]{Lemma}

\newtheorem{cor}[thm]{Corollary}

\usepackage{amssymb,latexsym,amsmath,epsfig}
\usepackage{graphicx,relsize}
\usepackage[numbers,square]{natbib}
\usepackage{mathtools}
\usepackage{mathrsfs}               

\newcommand{\mise}{mis\`{e}re }

\newcommand{\R}{\mathcal{R}}
\renewcommand{\L}{\mathcal{L}}
\newcommand{\N}{\mathcal{N}}
\renewcommand{\P}{\mathcal{P}}

\renewcommand{\v}{\phantom{'} \vert \phantom{'}}

\newcommand{\lopt}{\boldsymbol{L}}
\newcommand{\ropt}{\boldsymbol{R}}
\newcommand{\bs}{\boldsymbol}

\usepackage{array,color}
\usepackage{colortbl}
\usepackage{tabularx}
\definecolor{gr}{gray}{0.6}

\graphicspath{{Images/}}

\usepackage{tikz}

\makeatletter
\def\imod#1{\allowbreak\mkern10mu({\operator@font mod}\,\,#1)}
\makeatother

\begin{document}
\title{Partizan Kayles and Mis\`ere Invertibility}

\author{Rebecca Milley\\
Grenfell Campus\\
Memorial University of Newfoundland\\
Corner Brook, NL, Canada}
\maketitle

\begin{abstract}
The impartial combinatorial game {\sc kayles} is played on a row of  pins, with players  taking turns removing either a single pin or two adjacent pins.  A natural partizan variation is to allow one player to remove only a single pin and the other only a pair of  pins.  
This paper develops a complete solution for {\sc partizan kayles} under mis\`ere play, including the mis\`ere monoid all possible sums of positions, and discusses its significance in the context of mis\`ere invertibility: the universe of {\sc partizan kayles} contains a position whose additive inverse is not its negative, and moreover, this position is an example of  a right-win game whose inverse is previous-win.

\end{abstract}


\section{Introduction}
In the  game of {\sc kayles}, two players take turns throwing a bowling ball at a row of pins.  A player either hits dead-on and knocks down a single pin, or hits in-between and knocks down a pair of adjacent pins.  
This game has been analyzed for both normal play (under which the player who knocks down the last pin wins) and mis\`ere play (when the player who knocks down the last pin loses) \cite{GuySmith,SiberC1992, Plamb1992}.  Since both players have the same legal moves, {\sc kayles} is an {\em impartial} game. Although there are several natural non-impartial or {\em partizan} variations, in this paper the rule set of {\sc partizan kayles} is as follows: the player `Left' can only knock down a single pin and the player `Right' can only knock down a pair of adjacent pins.  This game can be seen as a one-dimensional variant of {\sc domineering}, played on strips of squares (representing the rows of pins), with Left placing the bottom half of her vertical dominoes and Right placing his horizontal dominoes as usual.  For notational purposes, we will play `domineering-style', on $1\times n$ strips denoted $S_n$, with Left placing squares and Right placing dominoes, as illustrated in Figure \ref{kaylesexample}.

\begin{figure}[htp]
\begin{center}
\includegraphics[scale=0.45]{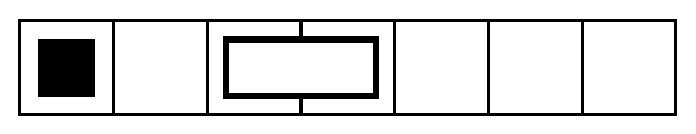}
\caption{A game of {\sc partizan kayles}: after one Left move and one Right move, the original position $S_6$ becomes   $S_1+S_3$.}
\label{kaylesexample}\end{center}
\end{figure}

%  Note that Left can always move, if the position is non-zero, while Right can move as long as the position is not a sum of single squares.  This means there are no non-zero left ends, and any Right end (a sum of single squares) will remain a Right end. Thus,  {\sc partizan kayles} is dead-ending. AND A PLACEMENT

This paper develops a complete solution for {\sc partizan kayles} under mis\`ere play.
%, including the mis\`ere monoid and partial order of all possible sums of positions.   
We will see that the set ({\em universe}) of {\sc partizan kayles} positions is remarkable for its unusual examples of invertibility.
%including its relevance to the  {\em Mis\`ere Invertibility Conjecture}, which was originally presented in the author's PhD thesis \cite{Milley} and has since been pursued by a number of mis\`ere-play researchers\footnote{The conjecture has been considered by Paul Dorbec, Richard Nowakowski, Paul Ottaway, Gabriel Renault, Eric Sopena, and Mike Weiemerskirch.}.

It is assumed that the reader is familiar with basic normal-play combinatorial game theory\footnote{A complete overview of normal-play combinatorial game theory can be found in \cite{AlberNW2007}.}.  A brief review of the necessary mis\`ere background is given in Section 1.1; a more detailed overview can be found in \cite{MilleyRenault}.  Section 2 establishes domination of options among {\sc partizan kayles} positions, and Section 3 uses this to show how every strip of length at least three reduces to a disjunctive sum of single squares and strips of length two.  Section 4 then gives the outcome and strategy for a general sum, including the {\em mis\`ere monoid} of the universe of {\sc partizan kayles} positions. Finally, Section 5 discusses the significance of this universe in the context of mis\`ere invertibility.

\subsection{Mis\`ere prerequisites}
%By convention, the players Left and Right are female and male, respectively.  
A game or {\em position} is defined in terms of its options: $G=\{G^{\lopt} \v G^{\ropt} \}$, where 
$G^{\lopt}$ is the set of positions $G^L$ to which Left can move in one turn, and similarly for $G^{\ropt}$.
  The simplest game is the zero game, $0=\{\cdot \v \cdot\}$, where the dot indicates an empty set of options.
The outcome function $o^-(G)$ gives the \mise outcome of a game $G$.  If $o^-(G)=\P$ we write $G\in \P^-$, where again the superscript  indicates that this is the outcome under mis\`ere play.
Thus, for example, $0\in \N^-$.
%\begin{figure}[htp]
%\begin{center}
%\begin{tabular}{c}
%$\L$\\
%$\diagup \quad \quad \diagdown$\\
%$\P \quad \quad \; \;	\quad \N$\\
%$\diagdown \quad \quad \diagup$\\
%$\R$\\
%\end{tabular}
%\end{center} 
%\caption{The partial order of outcome classes.}
%\label{NPLR}
%\end{figure}

Many definitions from normal-play theory are used without modification for \mise games, including  disjunctive sum, equality, and inequality.  However, the normal-play {\em negative} of $G$, $-G =  \{-{G}^{\ropt} | 
{-G}^{\lopt}\}$, is instead called the {\em conjugate} of $G$, denoted $\overline{G}$, since in general we do not have $G+\overline{G}=0$ in mis\`ere play \cite{MesdaO2007}.  Thus $$\overline{G}=\{ \overline{{G}^{\ropt}} | 
\overline{{G}^{\lopt}\}}.$$ 
%
%Thus, for \mise games, 
%$$G= H \textrm{ if } o^-(G+X)= o^-(H+X) \textrm{ for all games } X,$$
%$$G\geq H \textrm{ if } o^-(G+X)\geq o^-(H+X) \textrm{ for all games } X.$$
%In normal-play, the \textit{negative} of a game is defined recursively as $-G=\{-{G}^{\ropt} | 
%{-G}^{\lopt}\}$, and is so-called because $G+ (-G)=0$ for all games $G$.  Under general mis\`{e}re play, however, this property holds only if $G$ is the zero game \cite{MesdaO2007}.  To avoid confusion and inappropriate cancellation, we write $\overline{G}$ instead of $-G$ and refer to this game as the {\em conjugate} of $G$.

In addition to the well-studied  {\em normal-play canonical form}, every position has a unique {\em mis\`ere canonical form} obtained by eliminating dominated options and bypassing reversible ones \cite{Siege}; note that the definitions of domination and reversibility are indeed dependant on the ending condition, since the definition of inequality is dependant on the ending condition.

Mis\`ere games are much more difficult to analyze than normal-play games (see \cite{Plamb2009}, for example). No position besides $\{\cdot \v \cdot\}$ is equal to $0$ \cite{MesdaO2007}; thus, no non-zero position has an additive inverse, and there is no easy test for equality or inequality of games.  This in turn means that instances of domination and reversibility are rare and hard to establish, so that we cannot take advantage of canonical forms in mis\`ere play as we do in normal play. 

 Some of these problems are mitigated by considering  restricted versions of equality and inequality  \cite{Plamb2005, PlambS2008}.  Let $\mathcal{U}$ be a set of games closed under disjunctive sum and followers (but not necessarily under conjugation). {\em Indistinguishability} or {\em equivalence (modulo} $\mathcal{U}${\em )} is defined by 
$$G\equiv H \textrm{ (mod } \mathcal{U} \textrm{) if } o^-(G+X)= o^-(H+X) \textrm{ for all games } X \in \mathcal{U},$$
while {\em  inequality (modulo} $\mathcal{U}${\em)} is defined by
$$G\geqq H \textrm{ (mod } \mathcal{U} \textrm{) if } o^-(G+X)\geq o^-(H+X) \textrm{ for all games } X \in \mathcal{U}.$$
The set $\mathcal{U}$ is called the {\em universe}.
 If $G\not \equiv H$ (mod $\mathcal{U}$) then $G$ and $H$ are said to be {\em distinguishable} modulo $\mathcal{U}$, and in this case there must be a game $X\in \mathcal{U}$ such that $o^-(G+X)\not = o^-(H+X)$.  If $G\not \geqq H$ (mod $\mathcal{U}$) and $G\not \leqq H$ (mod $\mathcal{U}$) then $G$ and $H$ are {\em incomparable} in $\mathcal{U}$.
 %, and we write $G||H$ (mod $\mathcal{U}$). 
 The symbol $\gneqq$ is used to indicate strict modular inequality.  
In this paper it is assumed that both $G$ and $H$ are contained in $\mathcal{U}$ when we compare them modulo $\mathcal{U}$.

Note that indistinguishability is a congruence relation.
Given a universe $\mathcal{U}$, we can determine the equivalence classes under indistinguishability modulo $\mathcal{U}$.  Since we may still not have inverses for every element, the classes  
 form a quotient monoid.  Together with the tetra-partition of
elements into the sets $\P^-$, $\N^-$, $\R^-$, and $\L^-$, this quotient is called the {\em \mise monoid} of the set $\mathcal{U}$,  denoted $\mathscr{M}_{\mathcal{U}}$ \cite{Plamb2005}. 

Indistinguishability and mis\`ere monoids have been successfully used to analyze various impartial \cite{PlambS2008} and  partizan games \cite{Allen2009, MilleyNO, McKayMN}.
 This paper, which summarizes a section of the author's PhD thesis \cite{Milley}, develops the monoid for the universe of {\sc partizan kayles} positions, and discusses its relevance to `restricted' (modulo $\mathcal{U}$) mis\`ere invertibility.  The game of {\sc partizan kayles} is a {\em placement game}, in that  players move by putting pieces on a board, and is thus also {\em dead-ending}, meaning that  once a player has no current move, that player will never have another move.  The universe of dead-ending games is introduced and explored in \cite{MilleyRenault}, and both this and the subuniverse of placement games are exciting areas of future mis\`ere research.

\section{Domination}
The goal of this section is to establish domination of moves in {\sc partizan kayles}, with the concluding and most important result (Corollary \ref{211})  being that $S_2 \gneqq S_1+S_1$ modulo this universe.
Recall that $S_n$ denotes a strip of length $n$. The disjunctive sum of $k$ copies of $G$ is denoted $kG$, so that, for example, $S_1+S_1=2S_1$.  Let $\mathcal{K}$ be the universe of {\sc partizan kayles} positions; that is, $\mathcal{K}$ is the set of all possible sums of positions of the form $S_n$.
 Note that $S_0=\{\cdot \v \cdot\} = {0}$ and $S_1=\{{0} \v \cdot\} ={1}$ (the normal-play canonical-form integers), but  this is not the case for higher values of $n$; for example, $S_2 = \{{1} \v {0}\} \not = {2}$.

It should be immediately apparent to any player of mis\`ere games that this version of {\sc kayles} is heavily biased in favour of Right: Left can always move, if the position is non-zero, while Right cannot move on any sum of single squares.  It is therefore not surprising that there are no left-win positions in this universe, as demonstrated in  Lemma \ref{noleft}. As a consequence, we know that if Left can win playing first in a position $G$ then $G \in \N^-$, and if Left can win playing second in $G$ then $G\in \P^-$. 

\begin{lem} \label{noleft}
If $G \in \mathcal{K}$ then $o^-(G)\not =\L$.
\end{lem}

\begin{proof}
If $G\in \mathcal{K}$ then $G$ is a sum of positions of the form $S_n$. Let $m$ be the total number of squares in $G$.
Note that each of Right's turns reduces the total number of free squares by 2 and each of Left's moves reduces the number by 1.  

If the total number $m$ is a multiple of 3 and Right plays first, then Left begins each turn with $3k+1$ free squares (for some $k\in \mathbb{N}$); in particular Left never begins a turn with zero free squares, and so can never run out of moves before Right.  This shows Right wins playing first, so $G\in \R^-\cup \N^-$. 

If $m\equiv 1$ (mod $3$) then Left playing first begins each turn with $3k+1$ free squares and Left playing second begins each turn with $3k+2$ free squares; in either case Left cannot run out of moves before Right by the same argument as above.  Here Right wins playing first or second so $G\in \R^-$.

Finally, if $m\equiv 2$ (mod $3$), then Left playing first necessarily moves the game to one in which the total number of squares is congruent to 1 modulo 3, and as shown above this is a right-win position.  Thus Left loses playing first and the game is in $\R^-$ or $\P^-$.
\end{proof}

Since single squares are so detrimental for Left, we might naively\footnote{This strategy is `naive' in that it does not always (or even usually) work for mis\`ere games; for example, the game $\{0\v\cdot\}$ is incomparable with the zero game in general mis\`ere play.}    suspect that Left should get rid of them as quickly as she can.  That is, given a position that contains an $S_1$, Left should do at least as well by playing in the $S_1$ as playing anywhere else.  This is indeed the case, as established in Corollary \ref{leftprefers}. The bulk of the work is done in Lemma \ref{Kstratlemma}.  Lemma \ref{noleft} (that is, the non-existence of left-win positions in $\mathcal{K}$) is used repeatedly without reference in the following proof.

\begin{lem} \label{Kstratlemma}
If $G\in \mathcal{K}$ then $G\geqq G^L+S_1$ (mod $\mathcal{K}$) for all $G^L \in G^{\bs{L}}$.
\end{lem}

\begin{proof}
We must show that  $o^-(G+X)\geq o^-(G^L+S_1+X)$, for any $X\in \mathcal{K}$, where $G^L$ is any Left option of $G$.  Since $G$ is already an arbitrary game in $\mathcal{K}$, it suffices to show  $o^-(G)\geq o^-(G^L+S_1)$.  To do so we will show that when $G^L+S_1$ is in $\N^-$, $G$ is also in $\N^-$, and that when $G^L+S_1$ is in $\P^-$, $G$ is also in $\P^-$.  If $G^L+S_1$ is in $\R^-$, then we trivially have $o^-(G)\geq o^-(G^L+S_1$).

Suppose $G^{L}+S_1 \in \N^-$, so that Left has a good first move in $G^{L}+S_1$.  If the good move is to $G^{L}+0=G^L$ then Left has the same good first move in $G$, and so $G\in \N^-$. Otherwise the good move is to $G^{LL}+S_1 \in \P^-$, for some Left option $G^{LL}$ of $G^L$;
but then by induction $o^-(G^L) \geq o^-(G^{LL}+S_1)$ and so $G^L\in \P^-$ and  $G\in \N^-$.

Now suppose $G^L+S_1\in \P^-$. We must show that $G\in \P^-$.  Since Right has no good first move in $G^{L}+S_1$, we have $G^{LR}+S_1 \in \N^-$  for every right option $G^{LR}$ of $G^L$.  So Left has a good first move in $G^{LR}+S_1$;  by induction the move to $G^{LR}$ is at least as good as any other, and so $G^{LR}\in \P^-$  for every right option  $G^{LR}$ of $G^L$.  We will see that there  exists a previous-win Left response $G^{RL}$ to every first Right move $G^R$, by finding a $G^{RL}$ that is equal to some $G^{LR}\in\P^-$. This gives a winning strategy for Left playing second in $G$, proving $G\in \P^-$.

Let $G^{R}$ be any right option.  If the domino placed by Right to move from $G$ to $G^R$ would not overlap the square placed by Left to move from $G$ to $G^{L}$, then Left can place that square now, achieving a position $G^{RL}$ equal to some $G^{LR}$, which we know to be in $\P^-$.  This is a good second move for Left in $G$, so $G\in \P^-$.  If Right's move from $G$ to $G^{R}$ {\em does} interfere with Left's move from $G$ to $G^{L}$, then we will see that Left can still move $G^R$ to a position equal to some $G^{LR}$.  If there are free squares adjacent to both sides of the domino Right places for $G^R$, then Left can respond to $G^R$ by playing in one of those squares, so that the resulting position $G^{RL}$ equal to a position $G^{LR}\in \P^-$.  This is illustrated in Figure \ref{case1}.

\begin{figure}[htp]
\begin{center}
\includegraphics[scale=0.45]{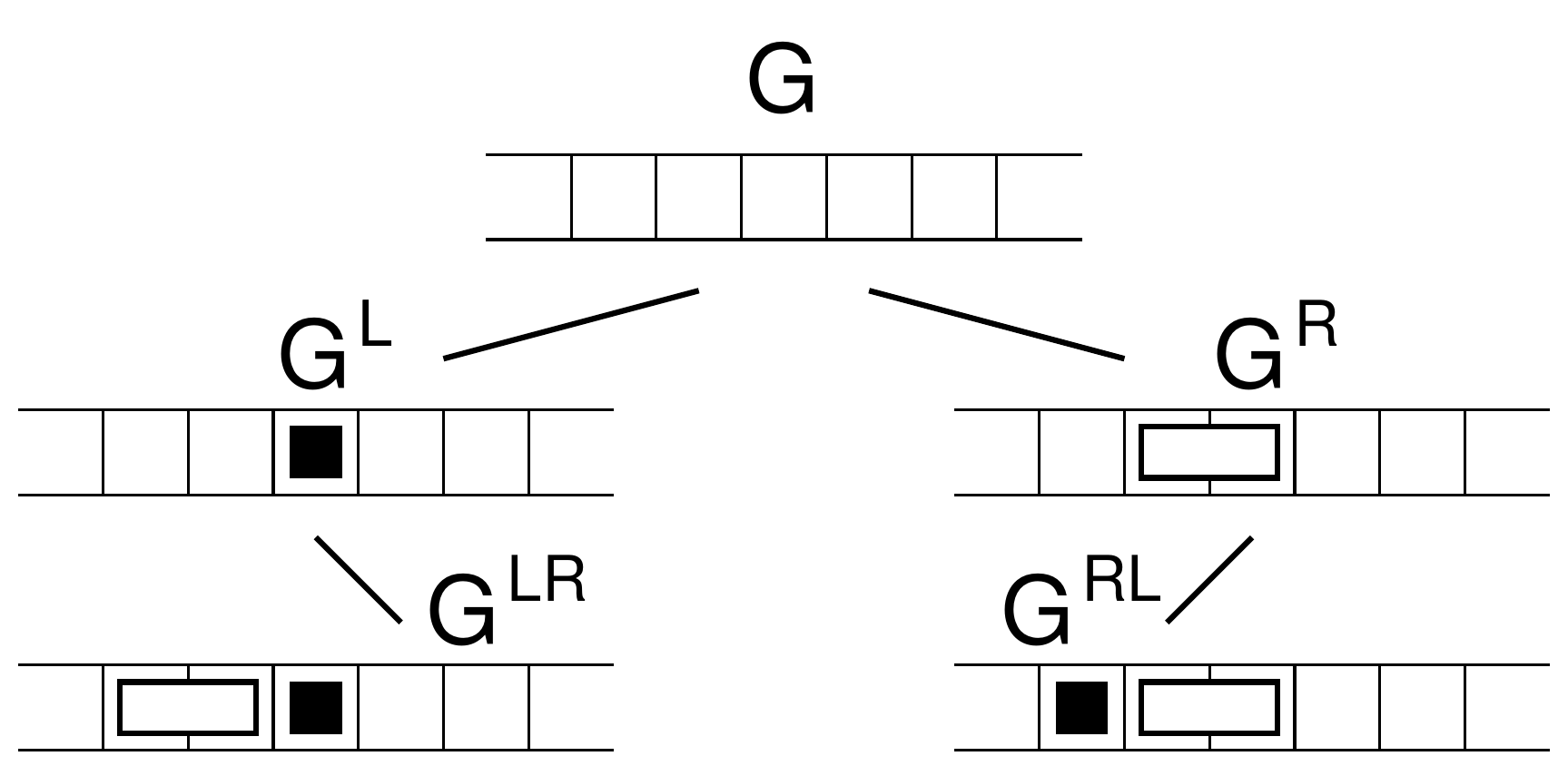}
\caption{Although the pieces for $G^L$ and $G^R$ overlap, Left  still has  an  option of $G^R$ that is equal to a right option of $G^L$.}
\label{case1}\end{center}
\end{figure}

If Left cannot so easily obtain such a position --- if there is not a free square on both sides of Right's domino --- then we have several cases to consider.

%Case 1: Left moves from $G$ to $G^L$ by placing a square at least two away from either end of some component $S_n$, $n\geq 5$, of $G$.  So $G=S_n+G'$ and $G^L= S_i+S_{n-i-1}+G'$, with $i\geq 2$. Since Right's domino placement from $G$ to $G^R$ would overlap Left's square, we have either $G^R=S_{i}+S_{n-i-2} +G'$ or $G^R=S_{i-1}+S_{n-i-1}+G'$. From the former Left can move to $ S_{i}+S_{n-i-3}+G'$, and from the latter Left can move to $S_{i-2}+S_{n-i-1}+G'$, both of which are previous-win because they are right options of $G^L$.

Case 1: The domino placed by Right to move from $G$ to $G^R$ is at the end of a component $S_n$, $n\geq 4$.  So $G=S_n+G'$ and $G^R=S_{n-2}+G'$. Since this domino interferes with  Left's move from $G$ to $G^L$, Left's move to $G^L$ must be  to place a square at the end or one away from the end of $S_n$:  $G^L =S_{n-1}+G'$ or $G^L= S_1+S_{n-2}+G'$.  If the former, then as above Left responds to $G^R$ by playing adjacent to Right, obtaining $G^{RL}=S_{n-3}+G'$, which is in $\P^-$ because it is a Right option of $G^L =S_{n-1}+G'$.  If the latter, then Left responds to $G^R$ by playing one away from the end of $S_{n-2}$, leaving the position $S_1+S_{n-4}+G'$, which is a Right option of $G^L=S_1+S_{n-2}+G'$, and is therefore in $\P^-$.
 
Case 2: The domino placed by Right to move from $G$ to $G^R$ is at the end of a component $S_3$. So $G=S_3+G'$ and $G^R=S_1+G'$.  If Left's move to $G^L$ is to play at the end of this $S_3$, then she simply plays in the $S_1$ now to obtain $G^{RL}=G'$, which is equal to a $G^{LR}$ and so is in $\P^-$.  If Left's move to $G^L$ is to play in the middle of $S_3$, then $G^L=S_1+S_1+G'$.  It cannot be that $G'$ is a sum of all $S_1$ positions, else $G^L+S_1$
 is right-win, and we are assuming it is previous-win.  So there is at least one component $S_n$ in $G'$ with $n\geq 2$.  Thus we can write $G'=S_n+G''$.  Left should respond to $G^R = S_1+S_n +G''$ by moving one away from the end of the $S_n$, to obtain $G^{RL} = S_1+S_1+S_{n-2}+G''$, which is in $\P^-$ because it is a right option of $G^L=S_1+S_1+G'=S_1+S_1+S_n+G''$. 

Case 3: The domino  placed by Right to move from $G$ to $G^R$ is in an $S_2$.  So $G=S_2+G'$ and $G^R=G'$.  Since this domino interferes with Left's move to $G^L$, we must have $G^L=S_1+G'$.  Again, it cannot be that $G'$ is a sum of all $S_1$ positions, since $G^L+S_1\in\P^-$.   So there is at least one component $S_n$ in $G'$ with $n\geq 2$. Left should respond to $G^R=G'=S_n+G''$ by playing one away from the end of $S_n$, so that she obtains the position $S_1+S_{n-2}+G''$, which is a Right option from $G^L=S_1+G'=S_1+S_n+G''$. 

In every case Left has a good second move in $G$ ($G^{RL}\in \P^-$) and so $G\in \P^-$, as required.
\end{proof}

As corollaries of this lemma we obtain both a general strategy for Left in {\sc partizan kayles} as well as the inequality $S_2 \gneqq 2S_2$.
  
\begin{cor} \label{leftprefers}
For any position $G\in \mathcal{K}$, if Left can win $G+S_1$ then Left can win by moving to $G$.
\end{cor}

\begin{proof}
Any other option of $G+S_1$ is of the form $G^L+S_1$, and by Lemma \ref{Kstratlemma}, $G^L+S_1$ is dominated by $G$.
\end{proof}

\begin{cor}\label{211}
$S_2 \gneqq S_1+S_1$ (mod $\mathcal{K}$).
\end{cor}

\begin{proof} $S_2 \geqq 2S_1$  follows directly from Lemma \ref{Kstratlemma} with $G=S_2$, since the only left option $G^L$ is $S_1$.  The inequality is strict because $S_2 \in \P^-$ while $2S_1 \in \R^-$.
\end{proof}

\section{Reduction}
Corollary \ref{211} is the key to the solution of {\sc partizan kayles}: it allows us, by establishing domination of options, to show that every strip $S_n$  `splits' into a sum of $S_1$ positions (single squares) and $S_2$ positions (dominoes). 
% Then, in Section 4, we determine the outcome of $kS_1+jS_2$.
 Theorem \ref{Kreduce} demonstrates the reduction.  
 Let us work through a few reductions by hand to gain some insight into this process.  These reductions are illustrated in Figure \ref{kaylesreduce}.

Trivially, $S_1$ and $S_2$ are already sums of single squares and dominoes.  In a strip of length $3$, Left has options to $S_2$ (playing at either end) and $S_1+S_1$ (playing in the middle).  Right has only one option, to $S_1$.  These are precisely the options of $S_1+S_2$; both games are equal to $\{S_2,2S_1\v S_1\}$ (which, in canonical form, is the game $\{S_2\v S_1\}$, by Corollary \ref{211}).  Thus, $S_3 = S_1+S_2$.

In a strip of length $4$, Left's options are to $S_3$ or $S_1+S_2$; as just established, these are equivalent.  Right's options are to $S_2$ or $S_1+S_1$, and the second  dominates  the first by Corollary \ref{211}.  So $S_4 \equiv \{S_1+S_2 \v 2S_1\}$.  Compare this to the position $2S_1+S_2  = \{S_1+S_2, 3S_1 \v 2S_1\}$; they are equivalent because the first left option dominates the second.  Thus, $S_4 \equiv 2S_1+S_2$.  

Lastly, consider a strip of length $5$. Left's options are $S_4\equiv 2S_1+S_2$, $S_1+S_3\equiv 2S_1+S_2$, and $S_2+S_2$, which dominates the others.  Right's options are $S_3$ and $S_1+S_2$, which are equivalent.  So $S_5 \equiv \{2S_2 \v S_1+S_2\}$.  This is the same as the position $S_1+2S_2$, as Left's move to $2S_2$ dominates here and Right's only move is to $S_1+S_2$.  That is, $S_5 \equiv S_1+2S_2$.  

\begin{figure}[htp]
\begin{center}
\includegraphics[scale=0.45]{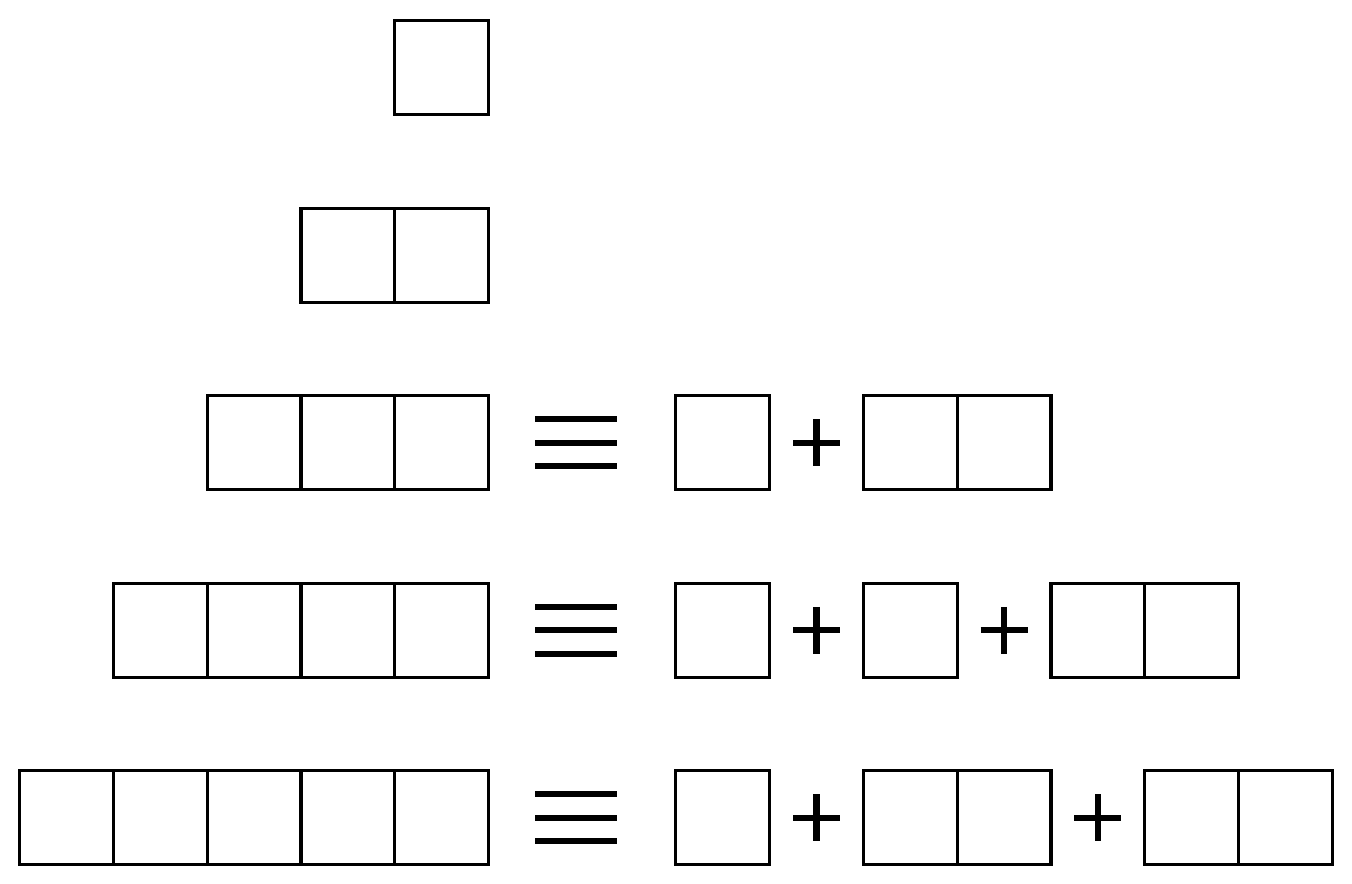}
\caption{Reduction of $S_n$ into a sum of $S_1$ and $S_2$ positions, for $n=1,\ldots,5$.}
\label{kaylesreduce}\end{center}
\end{figure}

If we were to continue with $S_6$, $S_7$, and $S_8$, we would observe a pattern based on the congruency of $n$ modulo $3$.  
The reductions for longer strips use the same logic as the cases for $n=1,\ldots, 5$, and indeed the general inductive proof follows a similar method, of considering the possible options and removing those dominated via Corollary \ref{211}.  We now begin the general argument for reducing any $S_n$.  Lemma \ref{reducelemma} serves to tidy up the proof of Theorem \ref{Kreduce}.

\begin{lem}\label{reducelemma}
If $k,j \in \mathbb{N}$  then $kS_1+jS_2 \equiv \{(k-1)S_2 +jS_2 \v kS_1 + (j-1)S_2\}$ (mod $\mathcal{K})$.
\end{lem}

\begin{proof}
Left's only moves in $kS_1+jS_2$ are to bring an $S_1$ to zero or an $S_2$ to an $S_1$.  These moves give the options $(k-1)S_1+jS_2$ and $(k+1)S_1+(j-1)S_2$, respectively, and the second is dominated by the first because $S_2 \geqq 2S_1$.
Right has only one move up to symmetry --- play in an $S_2$ --- and so $kS_1+jS_2 \equiv \{(k-1)S_2 +jS_2 \v kS_1 + (j-1)S_2\}$ (mod $\mathcal{K})$, as claimed.
\end{proof}

\begin{thm}\label{Kreduce}
If $n\geq 3$ then
$$
S_n \equiv
\begin{cases}
kS_1+kS_2 \text{ (mod $\mathcal{K}$)}, & \text{if } n=3k,\\
(k+1)S_1+kS_2 \text{ (mod $\mathcal{K}$)}, & \text{if } n=3k+1,\\
kS_1+(k+1)S_2 \text{ (mod $\mathcal{K}$)}, & \text{if } n=3k+2.\\
\end{cases}
$$
\end{thm}

\begin{proof}  
By the lemma,  
%$kS_1+kS_2 \equiv \{(k-1)S_1+kS_2 \v kS_1 + (k-1)S_2\} (\textrm{mod } \mathcal{K}),$
%so that when $n=3k$
 it suffices to show that $S_n \equiv \{(k-1)S_1+kS_2 \v kS_1 + (k-1)S_2\}$ when $n=3k$, that $S_n\equiv \{kS_1+kS_2 \v \{(k+1)S_1+(k-1)S_2\}$ when $n=3k+1$, 
and that $S_n \equiv \{(k-1)S_1+(k+1)S_2 \v kS_1 + kS_2 \}$ when $n=3k+2$.  The proof is broken into these three cases.
Note that any left option of $S_n$ is of the form $S_i+S_{n-1-i}$, with $0\leq i \leq n-1$.  Similarly, any right option of $S_n$ is of the form $S_i+S_{n-2-i}$, with $0 \leq i \leq n-2$.
%We will see that, by induction, left options $S_i+S_{n-1-i}$ and $S_{i'}+S_{n-1-i'}$ are equivalent if $i\equiv i'$ (mod $3$).  The same holds for Right's options. Thus, in each of the following three cases, there are at most three distinct left options and three distinct right options, based on the three possibilities for $i$ modulo $3$.  Applying induction to these options splits them into sums of $S_1$ and $S_2$ positions, and we will see that in every case 
%  Corollary \ref{211} can be used to show that one option dominates the others.  In this way we find that in each case $S_n$ is equivalent to the desired position (as described in the previous paragraph).
\\ \\
\noindent {\em Case 1:} $n=3k$:

\noindent If $i=3j$ then $n-1-i=3k-3j-1=3(k-j-1)+2$, and $n-2-i=3k-3j-2 = 3(k-j-1)+1$.  By induction this gives left and right options
\begin{align*}
S_i+S_{n-1-i} &= jS_1+jS_2+(k-j-1)S_1+(k-j)S_2 \\
&= (k-1)S_1+kS_2= G^{L_1}; and\\
S_i+S_{n-2-i} &= jS_1+jS_2+(k-j)S_1+(k-j-1)S_2 \\
&= kS_1+(k-1)S_2=G^{R_1}.\\
\end{align*}
If $i=3j+1$ then $n-1-i=3k-3j-2=3(k-j-1)+1$ and $n-2-i=3k-3j-3 =3(k-j-1)$.  By induction,
\begin{align*}
S_i+S_{n-1-i} &= (j+1)S_1+jS_2+(k-j)S_1+(k-j-1)S_2 \\
&= (k+1)S_1+(k-1)S_2= G^{L_2};\\
S_i+S_{n-2-i} &= (j+1)S_1+jS_2+(k-j-1)S_1+(k-j-1)S_2 \\
&= kS_1+(k-1)S_2= G^{R_2}.\\
\end{align*}
 If $i=3j+2$ then $n-1-i=3k-3j-3=3(k-j-1)$ and $n-2-i= 3k-3j-4 =3(k-j-2)+2$, so by induction we have 
\begin{align*}
S_i+S_{n-1-i} &= jS_1+(j+1)S_2+(k-j-1)S_1+(k-j-1)S_2 \\
&= (k-1)S_1+kS_2=G^{L_3};\\
S_i+S_{n-1-i} &= jS_1+(j+1)S_2+(k-j-2)S_1+(k-j-1)S_2 \\
&= (k-2)S_1+kS_2 =G^{R_3}.\\
\end{align*}

 Left has only two distinct options: either $G^{L_1}=G^{L_3}=(k-1)S_1+kS_2$ (obtained by moving to $S_i+S_{n-1-i}$ with any $i\equiv 0,2$ (mod 3)), or $G^{L_2}=(k+1)S_1+(k-1)S_2$ (obtained by moving to $S_i+S_{n-1-i}$ with any $i \equiv 1$ (mod 3)).  By Corollary \ref{211}, $G^{L_2}$ is dominated by $G^{L_1}$.
%since $S_2\geqq 2S_1$, we see that $(k+1)S_1+(k-1)S_2$ is dominated by $(k-1)S_1+kS_2$.
Similarly, Right's options are  $G^{R_1}=G^{R_2}=kS_1+(k-1)S_2$ or $G^{R_3}=(k-2)S_1+kS_2$, and the latter is dominated by the former.  We conclude that, when $n=3k$,
$$S_n \equiv  \{ (k-1)S_1+kS_2  \v kS_1+(k-1)S_2\}  \equiv kS_1+kS_2.$$

\noindent {\em Case 2:} $n=3k+1$:
\\
In this case, by similar arguments and computations, we find
%\noindent If $i=3j$ then $n-1-i=3k+1-3j-1=3(k-j)$, and $n-2-i=3k+1-3j-2 = 3(k-j-1)+2$.  By induction we have 
%\begin{align*}
%S_i+S_{n-1-i} &= jS_1+jS_2+(k-j)S_1+(k-j)S_2 \\
%&= kS_1+kS_2,\\
%S_i+S_{n-2-i} &= jS_1+jS_2+(k-j-1)S_1+(k-j)S_2 \\
%&= (k-1)S_1+kS_2.\\
%\end{align*}
%If $i=3j+1$ then $n-1-i=3k+1-3j-2=3(k-j-1)+2$ and $n-2-i=3k+1-3j-3 =3(k-j-1)+1$.  By induction,
%\begin{align*}
%S_i+S_{n-1-i} &= (j+1)S_1+jS_2+(k-j-1)S_1+(k-j)S_2 \\
%&= kS_1+kS_2,\\
%S_i+S_{n-2-i} &= (j+1)S_1+jS_2+(k-j)S_1+(k-j-1)S_2 \\
%&= (k+1)S_1+(k-1)S_2.\\
%\end{align*}
% If $i=3j+2$ then $n-1-i=3k+1-3j-3=3(k-j-1)+1$ and $n-2-i= 3k+1-3j-4 =3(k-j-1)$, so by induction we have 
%\begin{align*}
%S_i+S_{n-1-i} &= jS_1+(j+1)S_2+(k-j)S_1+(k-j-1)S_2 \\
%&= kS_1+kS_2,\\
%S_i+S_{n-1-i} &= jS_1+(j+1)S_2+(k-j-1)S_1+(k-j-1)S_2 \\
%&= (k-1)S_1+kS_2.\\
%\end{align*}
%In this case,
Left's only move is to $kS_1+kS_2$, while Right has an option to $(k-1)S_1+kS_2$  dominated by an option to $(k+1)S_1+(k-1)S_2$. Thus, if $n=3k+1$ then
$$S_n \equiv  \{ kS_1+kS_2  \v (k+1)S_1+(k-1)S_2\}  \equiv (k+1)S_1+kS_2.$$

\noindent {\em Case 3:} $n=3k+2$:
 \\
%\noindent If $i=3j$ then $n-1-i=3k+2-3j-1=3(k-j)+1$, and $n-2-i=3k+2-3j-2 = 3(k-j)$.  By induction we have 
%\begin{align*}
%S_i+S_{n-1-i} &= jS_1+jS_2+(k-j+1)S_1+(k-j)S_2 \\
%&= (k+1)S_1+kS_2,\\
%S_i+S_{n-2-i} &= jS_1+jS_2+(k-j)S_1+(k-j)S_2 \\
%&= kS_1+kS_2.\\
%\end{align*}
%If $i=3j+1$ then $n-1-i=3k+2-3j-2=3(k-j)$ and $n-2-i=3k+2-3j-3 =3(k-j-1)+2$.  By induction,
%\begin{align*}
%S_i+S_{n-1-i} &= (j+1)S_1+jS_2+(k-j)S_1+(k-j)S_2 \\
%&= (k+1)S_1+kS_2,\\
%S_i+S_{n-2-i} &= (j+1)S_1+jS_2+(k-j-1)S_1+(k-j)S_2 \\
%&= kS_1+kS_2.\\
%\end{align*}
% If $i=3j+2$ then $n-1-i=3k+2-3j-3=3(k-j-1)+2$ and $n-2-i= 3k+1-3j-4 =3(k-j-1)+1$, so by induction we have 
%\begin{align*}
%S_i+S_{n-1-i} &= jS_1+(j+1)S_2+(k-j-1)S_1+(k-j)S_2 \\
%&= (k-1)S_1+(k+1)S_2,\\
%S_i+S_{n-1-i} &= jS_1+(j+1)S_2+(k-j)S_1+(k-j-1)S_2 \\
%&= kS_1+kS_2.\\
%\end{align*}
In this case,
Left has a move to $(k+1)S_1+kS_2$ that is dominated by a move to $(k-1)S_1+(k+1)S_2$, while Right's only option is $kS_1+kS_2$. Thus, if $n=3k+2$ then
$$S_n \equiv  \{ (k-1)S_1+(k+1)S_2  \v kS_1+kS_2\}  \equiv kS_1+(k+1)S_2.$$
\end{proof}

\section{Outcome and strategy}
We have shown that every strip $S_n$ splits into a sum of single squares and dominoes.  This makes  analysis of the {\sc partizan kayles} universe much more manageable; we need only determine the outcome of a sum of any number of single squares and dominoes.   One trivial observation is that if there are more single squares than dominoes, then Left will not be able to win, as Right can eliminate all of `his' pieces before Left can run out of single squares.  That is, if $k>j$ then $o^-(kS_1+jS_2) = \R^-$.  Another immediate result is the outcome when there are exactly as many single squares as dominoes: the players are forced\footnote{The players are `forced' under optimal play, because Left will always choose to play in an $S_1$ over an $S_2$, by Corollary \ref{leftprefers}.} into a Tweedledum-Tweedledee situation where the first player runs out of moves first.  Thus, if $k=j$ then $o^-(kS_1+jS_2)=o^-(kS_1+kS_2)=\N^-$.  The outcome in the remaining case, when $k<j$, turns out to be dependant on the congruence of the total number of (not necessarily single) squares, modulo $3$; that is, it depends on the value of $k+2j$ (mod $3$).
%  This is not surprising in light of the very first result of this section, Lemma \ref{Ksum}.  Indeed, the following theorem has that lemma as a consequence. 

\begin{thm} \label{12outcome}
For positive integers $k$ and $j$,
$$
o^-(kS_1+jS_2) =
\begin{cases}
\N^-, & \text{if } k=j, \text{or if } k<j \text{ and } k+2j\equiv 0 \text{(mod $3$)},\\
\R^-, & \text{if } k>j, \text{or if } k<j \text{ and } k+2j\equiv 1 \text{(mod $3$)},\\
\P^-, & \text{if } k<j \text{ and } k+2j\equiv 2 \text{(mod $3$)}.
\end{cases}
$$
\end{thm}

\begin{proof}
Lemma \ref{reducelemma} states that 
$$kS_1+jS_2 \equiv \{(k-1)S_1+jS_2 \v kS_1+(j-1)S_2\}.$$
Let $G=kS_1+jS_2$.  We can prove each case by applying induction to $G^L=(k-1)S_1+jS_2$ and $G^R= kS_1+(j-1)S_2$.

If $k=j$ then Left's option is in $\P^-$ since $(k-1)+2k=3k-1$, and Right's option is in $\R^-$ since $k>k-1$.  So $G\in \N^-$.

If $k>j$ then $G^L \in \N^- \cup R^-$ and $G^R\in \R^-$, so  $G\in \R^-$.

If $k<j$ and $k+2j\equiv 0$ (mod 3) then $G^L \in \P^-$ because $k-1+2j \equiv 2$ (mod 3), while $G^R \in \R^-$ because $k+2j-2\equiv 1$ (mod 3). Thus $G\in \N^-$.

If $k<j$ and $k+2j\equiv 1$ (mod 3) then $G^L \in \N^-$ because $k-1+2j\equiv 0$ (mod 3), and $G^R \in \P^-$ because $k+2j-2\equiv 2$ (mod 3).  This confirms $G \in \R^-$.

Finally, if $k<j$ and $k+2j \equiv 2$ (mod 3) then $G^L \in \R^-$ because $k-1+2j\equiv 1$ (mod 3), and $G^R \in \N^-$ because $k+2j-2 \equiv 0$ (mod 3).  Thus $G \in \P^-$.
\end{proof}

As an immediate corollary we can prove what might be intuitively guessed about this universe: a single square and a single domino `cancel each other out'.  Essentially, we can think of a single square as one move for Left and a single domino as one move for Right.  
Things are more complicated  when only dominoes are present, because Left must then play in a domino, but this way of thinking works when at least one of each exists.  Corollary \ref{zerocor} has a very nice obvious consequence, which is given as the following corollary: any strip of length a multiple of $3$ is equivalent to zero.

\begin{cor} \label{zerocor} $S_1+S_2\equiv 0$ (mod $\mathcal{K}$). \end{cor}

\begin{proof}
Let $X\equiv kS_1+jS_2$ be any Kayles sum.  Then $o^-(X+S_1+S_2)=o^-[(k+1)S_1+(j+1)S_2] = o^-(kS_1+jS_2)$, by Theorem \ref{12outcome}, since
$$k=j \Leftrightarrow k+1=j+1,$$
$$k>j \Leftrightarrow k+1>j+1, \textrm{ and}$$
$$k+2j\equiv (k+1)+2(j+1) (\textrm{mod } 3).$$
\end{proof}

\begin{cor} If $n\equiv 0$ (mod $3$) then $S_n \equiv 0$ (mod $\mathcal{K}$).\end{cor}
\begin{proof} This is clear from Theorem \ref{Kreduce} and the previous corollary, since if $n=3k$ then $S_n$ reduces to $kS_1+kS_2=k(S_1+S_2)\equiv 0$ (mod $\mathcal{K}$).
\end{proof}

  With Theorem \ref{12outcome} its corollaries, we can  identify the mis\`ere monoid of {\sc partizan kayles}.  Every position in $\mathcal{K}$ is of the form $kS_1$, for an integer $k$, where $kS_1=|k|S_2$ if $k<0$.  The monoid is thus a group, isomorphic to the integers under addition.  We have
$$\mathscr{M}_{\mathcal{K}} = \langle  0,S_1,S_2 \v S_1+S_2=0 \rangle,$$
with outcome partition
$$\N^-=\{kS_2 \v k\geq0,k\equiv 0 (mod\; 3) \}, \P^-=\{kS_2 \v k>0,k\equiv 1 (mod\; 3)\},$$
$$ \R^- = \{kS_1,jS_2 \v k>0, j>0, j\equiv 2 (mod\; 3)\}, \L^-=\emptyset.$$

It would be nice to go a few  steps further and answer the following questions.

\begin{enumerate}
\item Can we `look' at a general sum of strips and determine the outcome, without having to first reduce the position to single squares and dominoes?
\item Can we determine the optimal move for a player when he or she has a winning strategy?
\end{enumerate}

The next theorem precisely answers question 1, by describing the outcome of a general Kayles position {\em without} directly computing its reduction into $S_1$ and $S_2$ pieces.  We must simply compare the number of pieces of length congruent to $1$ modulo $3$ to the number of those congruent to $2$ modulo $3$.  In fact, there is no new argument here: this is a compression of the two steps already discussed --- the reduction into $S_1$ and $S_2$ pieces (Theorem \ref{Kreduce}) and the outcome of $kS_1+jS_2$ (Theorem \ref{12outcome}).

\begin{thm} \label{xy}
If $G$ is a {\sc partizan kayles} position with $x$ pieces of length  $1$ modulo $3$ and $y$ pieces of length $2$ modulo $3$, then
$$o^-(G)=\begin{cases}
\N^-, & \text{if } x=y, \\ &\text{or if } x<y \text{ and } x+2y\equiv 0 \text{(mod $3$)};\\
\R^-, & \text{if } x>y, \\& \text{or if } x<y \text{ and } x+2y\equiv 1 \text{(mod $3$)};\\
\P^-, & \text{if } x<y \text{ and } x+2y\equiv 2 \text{(mod $3$)}.
\end{cases}
$$
\end{thm}

Finally, Theorem \ref{howtowin} answers our second question, of most interest to anyone actually playing {\sc partizan kayles}:  {\em how} do you win a general non-reduced {\sc partizan kayles} position, when you can? The winning moves described below can be confirmed using Theorem \ref{xy}

\begin{thm} \label{howtowin}
If Left can win a {\sc partizan kayles} position, then she can win playing at the end of a strip of length $1$ modulo $3$, when possible, or the end of a strip of length $2$ modulo $3$, otherwise.
If Right can can win a {\sc partizan kayles} position, then he can win playing at the end of a strip of length $2$ modulo $3$, when possible, or one away from the end of a strip of length $1$ modulo $3$, otherwise.
\end{thm}

\section{Discussion:  mis\`ere invertibility}
Within a  universe $\mathcal{U}$, a game $G$ may satisfy $G+\overline{G}\equiv 0$ (mod $\mathcal{U}$), and then $G$ is said to be {\em invertible} modulo $\mathcal{U}$.  For example, normal-play canonical-form numbers are invertible modulo the universe of all such positions \cite{MilleyRenault}.  If $G+\overline{G}\not \equiv 0$ (mod $\mathcal{U}$), it is tempting to say that $G$ is {\em not invertible} modulo $\mathcal{U}$ --- but once again mis\`ere games surprise us!  It is possible for some other position $H\not \equiv \overline{G}$ (mod $\mathcal{U}$) to satisfy $G+H\equiv 0$ (mod $\mathcal{U}$); that is, $G$ may have an additive inverse that is not its conjugate. The universe of {\sc partizan kayles} is the only known partizan example of such a situation: here we have $S_1+S_2\equiv 0$ (mod $\mathcal{K}$), with $S_1\not \equiv \overline{S_2}$ and $S_2 \not \equiv \overline{S_1}$. 

In fact, these comparisons are not even defined, as $\overline{S_1}$ and $\overline{S_2}$ do not occur in the universe $\mathcal{K}$. $\overline{S_1}$ would be a position in which Left has no move and Right has one move (that is, $\overline{S_1} = \{\cdot \v 0\} = -1$), and $\overline{S_2}$ would be a position in which Left can move to $0$ and Right can move to $-1$.  Even if we generalize the definition of equivalence to allow $G$ and $H$ to be compared modulo $\mathcal{U}$ without requiring that both are in $\mathcal{U}$, we do not obtain $S_1\equiv \overline{S_2}$; there is actually no universe\footnote{Assuming that any universe must contain zero, $S_1$ and $\overline{S_2}$ are always distinguishable.} in which $S_1$ and $\overline{S_2}$ are equivalent, since $S_1\in \R^-$ and $\overline{S_2} \in \P^-$.
This brings us to another oddity of {\sc partizan kayles}: there is a position in $\R^-$ whose additive inverse is in $\P^-$.
 There is no other known instance of an inverse pair having `asymmetric' outcomes in this way. It is likely  a symptom of the fact  that $\mathcal{K}$ is not closed under conjugation.

In \cite{Milley}, it was conjectured that $G+H\equiv 0$ (mod $\mathcal{U}$) implies $H \equiv \overline{G}$ (mod $\mathcal{U}$) whenever $\mathcal{U}$ is closed under conjugation. 
 This, however, is false, as an impartial counterexample appears in \cite{PlambS2008} (appendix A.6).  It was already known in \cite{Milley} (by the results presented here) that the stronger statement, removing the closure condition, is false.  The question now is whether a still weaker statement can be shown true: is there some condition on the universe $\mathcal{U}$ so that $G$ being invertible implies $G+\overline{G}\equiv 0$ (mod $\mathcal{U}$)? Is there a condition on the specific game $G$ that guarantees the invertibility of $G$? Can we find more counterexamples (so far there is only one) to the original conjecture of [3]?

Without Plambeck's theory of indistinguishability (equivalence), no non-zero game is invertible under mis\`ere play.  We now have a meaningful concept of additive inverses in restricted universes, but as {\sc partizan kayles} shows, invertibility for mis\`ere games is still strikingly different --- more subtle and less intuitive --- than invertibility for normal games. A better understanding of mis\`ere invertibility is a significant open problem in the growing theory of restricted mis\`ere play.

% It is an open question whether or not $G+H\equiv 0$  (mod $\mathcal{U}$) implies $H\equiv \overline{G}$ (mod $\mathcal{U})$, when $U$ is closed under conjugation.  

% shows that this implication fails for an `asymmetric' universe: the set of {\sc partizan kayles} positions has $G+H\equiv 0$ and $H\not \equiv \overline{G}$, but in fact $\overline{G}$ is not even in the universe.  For the purposes of this paper, the term {\em invertibility}  will specifically refer to $G$ and $\overline{G}$; thus, a game $G$ will be said to be {\em not invertible} modulo $\mathcal{U}$ if it is shown that $G+\overline{G}\not \equiv 0$ (mod $\mathcal{U}$), even though it is possible that $G$ may have some other additive inverse in $\mathcal{U}$.  We can partially justify this convention with the following conjecture.

%\begin{conj} \label{invertconj} 
%If $\mathcal{U}$ is closed under conjugates then $G+H\equiv 0$ (mod $\mathcal{U}$) implies $H\equiv \overline{G}$ (mod $\mathcal{U}$).
%\end{conj}

%When a game $H$ is invertible  modulo $\mathcal{U}$, we have $G\geqq H$  if and only if $G+\overline{H} \geqq 0$.  To see this, note that $G\geqq H$ (mod $\mathcal{U}$) if and only if $o^-(G+X)\geq o^-(H+X)$ for any game $X\in \mathcal{U}$. In particular, this holds if and only if $o^-(G+\overline{H}+X)\geq o^-(H+\overline{H}+X) = o^-(X)$; that is, if and only if $G+\overline{H}\geqq 0$.  
%

\end{document}